\documentclass[a4paper,reqno,12pt]{amsart} 
\usepackage{amsmath} 

\usepackage{amssymb}      

\usepackage{yhmath}
\usepackage{mathdots}
\usepackage{MnSymbol}

\usepackage{amsthm}      

\usepackage{tikz,amsmath}
\usetikzlibrary{arrows}

\usepackage{epsfig}      

\usepackage{graphicx}
\usepackage[normalem]{ulem}
\usepackage[all,line,
]{xy} 
\CompileMatrices 

\newcommand{\Ad}{\operatorname{Ad}}

\newcommand{\id}{\operatorname{id}}

\newcommand{\Br}{\operatorname{Br}}

   \theoremstyle{plain}
   \newtheorem{thm}{Theorem}[section]
   \newtheorem{prop}[thm]{Proposition}
   \newtheorem{lemma}[thm]{Lemma}  
   \newtheorem{cor}[thm]{Corollary}
   \theoremstyle{definition}

   \newtheorem{example}[thm]{Example}
   \theoremstyle{remark}
   
   \newtheorem{remark}[thm]{Remark}

\newtheorem{q}[thm]{Question}

\usepackage{mdframed,xcolor}

\definecolor{mybgcolor}{gray}{0.8}
\definecolor{myframecolor}{rgb}{.647,.129,.149}

\mdfdefinestyle{mystyle}{
  usetwoside=false,
  skipabove=0.6em plus 0.8em minus 0.2em,
  skipbelow=0.6em plus 0.8em minus 0.2em,
  innerleftmargin=.25em,
  innerrightmargin=0.25em,
  innertopmargin=0.25em,
  innerbottommargin=0.25em,
  leftmargin=-.75em,
  rightmargin=-0em,
  topline=false,
  rightline=false,
  bottomline=false,
  leftline=false,
  backgroundcolor=mybgcolor,
  splittopskip=0.75em,
  splitbottomskip=0.25em,
  innerleftmargin=0.5em,
  leftline=true,
  linecolor=myframecolor,
  linewidth=0.25em,
}

\newmdenv[style=mystyle]{important}

\usepackage{lipsum}

   \numberwithin{equation}{section}

        \date{\today}

\title[Phase transition]{Phase transition in the CAR algebra}
\author{Klaus Thomsen}

\date{\today}

\email{matkt@math.au.dk}
\address{Department of Mathematics, Aarhus University, Ny Munkegade, 8000 Aarhus C, Denmark}

\begin{document}

\maketitle
\begin{center} \emph{I mindet om Ola Bratteli}
\end{center}

\bigskip

\begin{abstract} The paper develops a method to construct one-parameter groups of automorphisms on a UHF algebra with a prescribed field of KMS states.
 
\end{abstract}

\section{Introduction} 

Let $\beta$ be a real number. A state $\omega$ of a $C^*$-algebra $A$ is a $\beta$-KMS state for a given continuous one-parameter group of automorphisms $\alpha = (\alpha_t)_{t\in \mathbb R}$ on $A$ when the identity
\begin{equation}\label{KMS}
\omega(a b) = \omega(b \alpha_{i\beta}(a))
\end{equation}
holds for all elements $a$ and $b$ that are analytic for $\alpha$, cf. \cite{BR}.
 Due to the interpretation of the KMS condition in models from quantum statistical mechanics and in particular the Ising model, the early history of the condition contains much work on KMS states for one-parameter groups on the CAR $C^*$-algebra and more general UHF algebras, \cite{Ar1}, \cite{Ar2}, \cite{PS},\cite{S1},\cite{S2},\cite{S3},\cite{S4}. None of these papers contain examples with more than one $\beta$-KMS state for each $\beta$, but Theorem 6.2.48 and Theorem 6.2.49 in Bratteli and Robinsons monograph \cite{BR} describe examples of one-parameter groups on the CAR algebra for which there are at least two $\beta$-KMS states for all sufficiently large values of $\beta$. In the notes to Section 6.2.6 in the first edition of \cite{BR} \footnote{Both theorems and the Notes and Remarks in Section 6.2.6 were edited for the second edition.}  the authors explain that the arguments underlying these examples are adopted from work by Peierls from 1936 and by Dobrushin from 1965 which both deal with classical statistical mechanics where $C^*$-algebras are absent. Subsequent work by Bratteli, Elliott, Herman and Kishimoto, \cite{BEH}, \cite{BEK}, based on the classification of AF algebras in terms of ordered K-theory, showed that very complicated variation with $\beta$ of the structure of $\beta$-KMS states can occur for actions on certain simple unital $C^*$-algebras; specifically, for the dual action restricted to a corner in a crossed product coming from a carefully chosen automorphism of an equally well-chosen AF algebra. But it was not until the work of Kishimoto in \cite{Ki} that one could begin to suspect that a more complicated KMS structure can occur in UHF algebras; see in particular Proposition 4.4 and Proposition 4.6 in \cite{Ki}. 
 
 It is the purpose with this paper to exhibit a method which makes it possible to piece together simplex cones to realize an almost arbitrary variation of $\beta$-KMS states for one-parameter groups of automorphisms on a UHF algebra. The actions we work with are the so-called generalized gauge actions which were studied in \cite{Th1} and \cite{Th2}, and parts of the constructions we perform are based on ideas from \cite{Th2}, but except for some basic standard references the present paper is self-contained. The constructions proceed in two steps, where the first step is an expansion of the methods from Section 9 of \cite{Th2} where they were used to build strongly connected graphs such that the gauge action on the corresponding graph $C^*$-algebra has the wildest possible variation with $\beta$ of the simplexes of KMS states. The inverse temperatures $\beta$ which can occur for a generalized gauge action on the $C^*$-algebra of a strongly connected graph are either all positive or all negative, while AF algebras admit one-parameter groups of automorphisms for which the simplexes of $\beta$-KMS states are non-trivial for all real $\beta$. For this reason it is necessary to develop the methods from \cite{Th2} further. We do this by using potentials that take both positive and negative values and show that ideas behind the methods from \cite{Th2} can be made to work using Bratteli diagrams only. This allows to control both the positive and the negative $\beta$, cf. Proposition \ref{13-10-18c}. In the second step we show that the structure of KMS states realized by a generalized gauge action on an AF algebra can also be realized by a generalized gauge action on any given UHF algebra, cf. Theorem \ref{25-09-18f}. A major tool for this is an approximate intertwining argument for inverse limits which is reminiscent of Elliott's approximate intertwining argument for inductive limits. In combination the two steps give a powerful method to construct generalized gauge actions on a UHF algebra with a prescribed structure of KMS states. In particular, it allows us to show that the extreme variation of KMS states which in \cite{Th2} was realized for the gauge action on the $C^*$-algebra of certain strongly connected graphs can also occur for generalized gauge actions on a UHF algebra, where the simplexes are now non-empty and different for all real $\beta$:

\begin{thm}\label{CAR2} Let $U$ be a UHF algebra. There is a generalized gauge action $\alpha$ on $U$ such that for all $\beta \neq 0$ the simplex $S^{\alpha}_{\beta}$ of $\beta$-KMS states for $\alpha$ is an infinite dimensional Bauer simplex and $S^{\alpha}_{\beta}$ is not affinely homeomorphic to $S^{\alpha}_{\beta'}$ when $\beta \neq \beta'$.
\end{thm}

The one-parameter groups of automorphisms we construct have presumably no physical significance, but it should be observed that they are well-behaved from a mathematical point of view; they are locally representable in the sense of Kishimoto, \cite{Ki}, and their generators are commutative normal derivations in the sense of Sakai, \cite{S2}.

\bigskip

\emph{Acknowledgement} I am grateful to Sergey Neshveyev for reminding me of the examples in \cite{BR} and to Derek Robinson for remarks on his book with Ola Bratteli.  And I thank Johannes Christensen for reading and commenting on earlier versions of the paper. The work was supported by the DFF-Research Project 2 `Automorphisms and Invariants of Operator Algebras', no. 7014-00145B.

\section{Generalized gauge actions on AF  $C^*$-algebras}

Let $\Gamma$ be a countable directed graph with vertex set $\Gamma_V$ and arrow set
$\Gamma_{Ar}$. For an arrow $a \in \Gamma_{Ar}$ we denote by $s(a) \in \Gamma_V$ its source and by
$r(a) \in \Gamma_V$ its range. In this paper we consider only graphs that are also row-finite without sinks, meaning that every vertex admits at least one and at most finitely many arrows, i.e.
$$
1 \leq \# s^{-1}(v) < \infty
$$
for all vertexes. A finite path in $\Gamma$ is an element
$\mu = (a_i)_{i=1}^{n}  \in \left(\Gamma_{Ar}\right)^{n}$ such that $r(a_i) = s(a_{i+1})$ for all
$i \leq n-1$ and $|\mu| = n$ is the length of $\mu$. We define the source $s(\mu)$ and range $r(\mu)$ of $\mu$ such that $s(\mu) = s(a_1)$ and $r(\mu) = r(a_n)$. A vertex $v \in \Gamma_V$ will be considered as
a finite path of length $0$. The $C^*$-algebra $C^*(\Gamma)$ of the graph $\Gamma$ was introduced in this generality in \cite{KPRR} and it is the universal
$C^*$-algebra generated by a collection $S_a, a \in \Gamma_{Ar}$, of partial
isometries and a collection $P_v, v \in \Gamma_V$, of mutually orthogonal projections subject
to the conditions that
\begin{enumerate}
\item[1)] $S^*_aS_a = P_{r(a)}, \ \forall a \in \Gamma_{Ar}$, and
\item[2)] $P_v = \sum_{a  \in s^{-1}(v)} S_aS_a^*, \ \forall v \in \Gamma_V$.
\end{enumerate} 
For a finite path $\mu = (a_i)_{i=1}^{|\mu|}$ of positive length we set
$$
S_{\mu} = S_{a_1}S_{a_2}S_{a_3} \cdots S_{a_{|\mu|}} \ ,
$$
while $S_{\mu} = P_v$ when $\mu$ is the vertex $v$. The elements $S_{\mu}S_{\nu}^*$, where $\mu,\nu$ range over all finite paths in $\Gamma$, span a dense $*$-sub-algebra in $C^*(\Gamma)$.

A function $F :  \Gamma_{Ar} \to \mathbb R$ will be called a potential on $\Gamma$ in the following. Using it we can define a continuous one-parameter group $\alpha^F = \left(\alpha^F_t\right)_{t \in \mathbb R}$ of automorphisms on $C^*(\Gamma)$ such that \label{alphaF}
$$
\alpha^F_t(S_a) = e^{i F(a) t} S_a
$$
for all $a \in \Gamma_{Ar}$ and 
$$
\alpha^F_t(P_v) = P_v
$$
for all $v \in \Gamma_V$. We call $\alpha^F$ a generalized gauge action.

 A Bratteli diagram $\Br$ is a special type of directed graph in which the vertex set $\Br_V$ is partitioned into level sets,
$$
\Br_V = \sqcup_{n=0}^{\infty} \Br_{n} \ ,
$$
with a finite number of vertexes in the $n$'th level $\Br_{n}$ such that $\Br_{0}$ consists of a single vertex $v_0$ which we call the top vertex in the following, and the arrows emitted from $\Br_n$ end in $\Br_{n+1}$, i.e. $r\left(s^{-1}(\Br_n)\right) \subseteq \Br_{n+1}$ for all $n$, \cite{Br}. Also, as is customary, we assume that a vertex in $\Br$ at most emits finitely many arrows, that $v_0$ is the only source in $\Br$ and that there are no sinks. We denote by $AF(\Br)$ the approximately finite dimensional (or  AF) $C^*$-algebra which was associated to $\Br$ by Bratteli in \cite{Br}. Let $\mathcal P_n$ denote the set of finite paths $\mu$ in $\Br$ emitted from $v_0$ and of length $n$, i.e. $s(\mu) = v_0$ and $|\mu| = n$. Let
$\mathcal P^{(2)}_n = \left\{ (\mu,\mu') \in \mathcal P_n \times \mathcal P_n : \ r(\mu) = r(\mu') \right\}$ and set
$$
E^n_{\mu,\mu'} = S_{\mu}S_{\mu'}^*  \ 
$$ 
when $(\mu,\mu') \in \mathcal P^{(2)}_n$. Then 
\begin{equation}\label{22-09-18}
\{E^n_{\mu,\mu'}\}_{(\mu,\mu') \in \mathcal P^{(2)}_n}
\end{equation}
are matrix units in $P_{v_0}C^*(\Br)P_{v_0}$ and $\sum_{\mu \in \mathcal P_n} E^n_{\mu,\mu} = P_{v_0}$. Let $\mathbb F_n$ be the $C^*$-sub-algebra of $P_{v_0}C^*(\Br)P_{v_0}$ spanned by the matrix units \eqref{22-09-18}. Then $\mathbb F_n \subseteq \mathbb F_{n+1}$ and
$$
P_{v_0}C^*(\Br)P_{v_0} = \overline{\bigcup_n \mathbb F_n} \ .
$$
Since the Bratteli diagram of the tower 
$\mathbb C \subseteq \mathbb F_1 \subseteq \mathbb F_2 \subseteq \mathbb F_3 \subseteq \cdots $ is $\Br$, it follows from \cite{Br} that $AF(\Br) \simeq P_{v_0}C^*(\Br)P_{v_0}$.

Let $ F : \Br_{Ar} \to \mathbb R$ be a potential and $\alpha^F$ the corresponding generalized gauge action of $C^*(\Br)$. Extend $F$ to finite paths $\mu = a_1a_2\cdots a_{|\mu|} \in \left(\Br_{Ar}\right)^{|\mu|}$ such that
$$
F(\mu) = \sum_{i=1}^{|\mu|} F(a_i) \ .
$$
Then
\begin{equation*}\label{22-09-18e}
\alpha^F_t(E^n_{\mu, \mu'}) = e^{it(F(\mu) - F(\mu'))} E^n_{\mu, \mu'} \ = \Ad e^{it H_n}\left(E^n_{\mu,\mu'}\right) \ ,
\end{equation*}
where 
$$
H_n = \sum_{\mu \in \mathcal P_n} F(\mu) E^n_{\mu,\mu} \in \mathbb F_n \ .
$$
In particular, by restriction the generalized gauge action $\alpha^F$ gives rise to a continuous one-parameter group of automorphisms on $AF(\Br)$ which we also denote by $\alpha^F$. In the following, by a generalized gauge action on an AF algebra $A$ we mean a continuous one-parameter group $\alpha$ on $A$ which is conjugate to the generalized gauge action $\alpha^F$ on $AF(\Br)$ for some Bratteli diagram $\Br$ and some potential $F$.

When $\alpha$ is a continuous one-parameter group of automorphisms on a unital separable $C^*$-algebra $A$ and $\beta$ a real number we denote the set of $\beta$-KMS states for $\alpha$ by $S^{\alpha}_{\beta}$, or $S^F_{\beta}$ when $A = AF(\Br)$ and $\alpha = \alpha^F$. Set
$$
\mathbb R^+ S^{\alpha}_{\beta} = \left\{ t \nu : \ t \in \mathbb R, \ t \geq 0, \ \nu \in S^{\alpha}_{\beta} \right\} \ ;
$$
the set of the positive linear functionals $\omega : A \to \mathbb C$ for which  \eqref{KMS} holds for all $\alpha$-analytic elements $a,b$. The topology we consider on $\mathbb R^+ S^{\alpha}_{\beta}$ is primarily the weak* topology and $S^{\alpha}_{\beta}$ is then a compact subset of $\mathbb R^+ S^{\alpha}_{\beta}$; in fact, a compact metrizable Choquet simplex by Theorem 5.3.30 in \cite{BR}. 

 %

\section{KMS states and projective matrix systems}

Let $\Br$ be a Bratteli diagram. A \emph{projective matrix system} over $\Br$ is a sequence $A^{(j)}, j =1,2,3,\cdots$, where 
$$
A^{(j)} = \left(A^{(j)}_{v,w}\right)_{(v,w) \in \Br_{j-1} \times \Br_j}
$$
is a non-negative real matrix over $\Br_{j-1} \times \Br_j$ subject to the condition that
$$
\left(A^{(1)}A^{(2)} \cdots A^{(k)}\right)_{v_0,w} \neq 0
$$
for all $w \in \Br_k$ and all $k \geq 1$. Let $\varprojlim_j A^{(j)}$ be the set of sequences
$$
\left(\psi^j\right)_{j=0}^{\infty} \in \prod_{j=0}^{\infty} [0,\infty)^{\Br_j} 
$$
for which $\psi^{j-1} = A^{(j)}\psi^{j}$ for all $j = 1,2,\cdots$. We consider $[0,\infty)^{\Br_j}$ as a closed subset of $\mathbb R^{\Br_j}$ and hence as a locally compact Hausdorff space, and we equip $\prod_{j=0}^{\infty} [0,\infty)^{\Br_j} $ with the corresponding product topology.

\begin{lemma}\label{21-09-18bb} 
\begin{enumerate}
\item[1)] $\psi^j_w  \leq  \left(\left(A^{(1)}A^{(2)} \cdots A^{(j)}\right)_{v_0,w}\right)^{-1}\psi^0$ for all $w \in \Br_j$ when $(\psi^j)_{j=0}^{\infty} \in \varprojlim_j A^{(j)}$.
\item[2)] $\varprojlim_j A^{(j)}$ is a closed locally compact convex cone in $ \prod_{j=0}^{\infty} [0,\infty)^{\Br_j}$; in fact, 
\begin{equation}\label{21-09-18cb}
\left\{ (\psi^j)_{j=0}^{\infty} \in \varprojlim_j A^{(j)}: \ \psi^0 \in K \right\}
\end{equation}
is a non-empty compact subset of $\varprojlim_j A^{(j)}$ for every non-empty compact subset $K \subseteq [0,\infty)$.
\end{enumerate}
\end{lemma}
\begin{proof} 1) follows from the observation that
$$
\psi^0 = \sum_{w \in \Br_j} \left(A^{(1)}A^{(2)} \cdots A^{(j)}\right)_{v_0,w}\psi^j_w \ 
$$
and 2) follows from 1) by Tychonoff's theorem.
\end{proof}

When $B = \left(B_{x,y}\right)_{(x,y) \in X \times Y}$ is a matrix over $X \times Y$, where $X$ and $Y$ are finite sets and all the entries in $B$ are positive numbers, define
\begin{equation}\label{phi}
\phi\left(B \right) = \min \left\{ \frac{B_{x,y} B_{x',y'}}{B_{x',y}B_{x,y'}} : \ x,x' \in X, \ y, y' \in Y \right\} \ .
\end{equation}

This quantity is introduced here because of the following well-known sufficient condition for the triviality of $\varprojlim_j A^{(j)}$. 

\begin{lemma}\label{30-10-18}
Let $\{A^{(j)}\}$ be a projective matrix system over $\Br$ such that $A^{(j)}_{v,w} > 0$ for $(v,w) \in \Br_{j-1} \times \Br_j$ and all $j$. 
Assume that $\sum_{j=1}^{\infty} \sqrt{\phi\left(A^{(j)}\right)} = \infty$. The map 
$$
 \varprojlim_j A^{(j)} \ni \psi \ \mapsto \ \psi^0
 $$
 is an affine homeomorphism $ \varprojlim_j A^{(j)} \simeq \mathbb R^+$.
 \end{lemma}
\begin{proof} This follows from Proposition 26.10 on page 280 in \cite{Wo}.
\end{proof}

For later reference we record the following simple observation.

\begin{lemma}\label{01-11-18b} Let $B = \left(B_{x,y}\right)_{(x,y) \in X \times Y}$ be a strictly positive  matrix over $X \times Y$ and $A = \left(A_{x,y}\right)_{(x,y) \in X \times Y}$ a non-negative matrix over $X \times Y$. Assume that 
$$
\frac{A_{x,y}}{B_{x,y}} \ \leq \ \epsilon \ < \ 1
$$
for all $(x,y) \in X \times Y$. Then $\phi(B + A) \geq (1+\epsilon)^{-2}\phi(B)$ and, if $B -A$ is strictly positive, $\phi(B-A) \geq (1-\epsilon)^2\phi(B)$.
\end{lemma}

Given a potential $F : \Br_{Ar} \to \mathbb R$ and a real number $\beta \in \mathbb R$ we define a projective matrix system $\{A^{(j)}(\beta)\}$ over $\Br$ such that
\begin{equation}\label{27-09-18a}
A^{(j)}(\beta)_{v,w} \ = \sum_{a \in s^{-1}(v) \cap r^{-1}(w)} e^{-\beta F(a)}
\end{equation}
when $(v,w) \in \Br_{j-1} \times \Br_j$. It follows from 2) in Lemma \ref{21-09-18bb} that the set
\begin{equation}\label{S-set}
\left\{ (\psi^j)_{j=0}^{\infty} \in  \varprojlim_j A^{(j)}(\beta) \ : \ \psi^0 = 1 \right\} \ 
\end{equation}
is a non-empty compact convex set for every $\beta \in \mathbb R$. 


\begin{prop}\label{21-09-18(90)b} Let $F : \Br_{Ar} \to \mathbb R$ be a potential. For all $\beta \in \mathbb R$ there is an affine homeomorphism from $\mathbb R^+ S^{F}_{\beta}$ onto $\varprojlim_j A^{(j)}(\beta)$ which maps $S^F_{\beta}$ onto the set \eqref{S-set}. 
\end{prop}
\begin{proof} Let $\omega \in \mathbb R^+ S^{F}_{\beta}$ and consider a vertex $v \in \Br_j$, $j \geq 1$. Let $(\mu,\mu') \in \mathcal P_j^{(2)}$ such that $r(\mu) = r(\mu') = v$. Then
\begin{equation*}
\begin{split}
&e^{\beta F(\mu')}\omega(  E^j_{\mu',\mu'}) = e^{\beta F(\mu')}\omega(  E^j_{\mu',\mu}E^j_{\mu,\mu'}) =   e^{\beta F(\mu')}\omega( E^j_{\mu,\mu'}\alpha^F_{i\beta}( E^j_{\mu',\mu})) \\
&=  e^{\beta F(\mu')}e^{-\beta(F(\mu')-F(\mu))} \omega( E^j_{\mu,\mu'} E^j_{\mu',\mu}) = e^{\beta F(\mu)}\omega(  E^j_{\mu,\mu}) \ ,
\end{split}
\end{equation*}
and we can therefore define $\psi(\omega)^j \in \mathbb R^{\Br_j}$ such that
$$
{\psi(\omega)}^j_v = e^{\beta F(\mu)} \omega\left( E^j_{\mu,\mu}\right)
$$
when $\mu \in \mathcal P_j$ and $r(\mu) = v$. Set $\psi(\omega)^0 = \omega(1)$ and note that
$$
\psi(\omega)^0 = \omega\left( \sum_{a \in \mathcal P_1} E^1_{a,a}\right) =  \sum_{a \in \mathcal P_1} e^{-\beta F(a)} {\psi(\omega)}^1_{r(a)} = \sum_{v \in \Br_1} A^{(1)}(\beta)_{v_0,v} {\psi(\omega)}^1_v \ .
$$
To calculate $\sum_{w \in \Br_j} A^{(j)}(\beta)_{v,w}\psi(\omega)^j_w$ when $j \geq 2$ choose $\nu \in \mathcal P_{j-1}$ such that $r(\nu) =v$. Then
\begin{equation*}
\begin{split}
&\psi(\omega)^{j-1}_v = e^{\beta F(\nu)} \omega(E^{j-1}_{\nu,\nu}) = e^{\beta F(\nu)} \sum_{a \in s^{-1}(v)} \omega(E^j_{\nu a, \nu a}) \\
&=  \sum_{w \in \Br_j} \sum_{a \in s^{-1}(v) \cap r^{-1}(w)}  e^{-\beta F(a)}e^{\beta F(\nu a)} \omega(E^j_{\nu a, \nu a})\\
& =  \sum_{w \in \Br_j} \sum_{a \in s^{-1}(v) \cap r^{-1}(w)} e^{-\beta F(a)} \psi(\omega)^j_w = \left(A^j(\beta)\psi(\omega)^j\right)_v \ .
\end{split}
\end{equation*}
This shows that we have a map $\Psi : \mathbb R^+ S^{F}_{\beta} \to \varprojlim_j A^{(j)}(\beta)$ defined such that $\Psi(\omega) =  \left(\psi(\omega)^j\right)_{j=0}^{\infty}$. To construct the inverse let $\psi = \left(\psi^j\right)_{j=0}^{\infty} \in \varprojlim_j A^{(j)}(\beta)$. For each $n = 1,2,3, \cdots$ there is a positive linear map $\omega^n_{\psi} : \mathbb F_n \to \mathbb C$ defined such that
$$
\omega^n_{\psi}\left( E^n_{\mu,\mu'}\right)  \ = \ \delta_{\mu,\mu'} e^{-\beta F(\mu)}\psi^n_{r(\mu)} \  . 
$$
Since
\begin{equation*}
\begin{split}
&\omega^{n+1}_{\psi}\left( \sum_{a \in s^{-1}(r(\mu))} E^{n+1}_{\mu a , \mu' a} \right) = \delta_{\mu,\mu'} \sum_{a \in s^{-1}(r(\mu))} e^{- \beta F(a)} e^{-\beta F(\mu)} \psi^{n+1}_{r(a)} \\
&  = \delta_{\mu,\mu'} \sum_{w \in \Br_{n+1}} e^{-\beta F(\mu)} A^{(n+1)}(\beta)_{r(\mu),w}   \psi^{n+1}_{w}  \\
&  = \delta_{\mu,\mu'} e^{-\beta F(\mu)} \psi^{n}_{r(\mu)}  \ \ = \ \ \omega^n_{\psi}(E^n_{\mu,\mu'}) \ ,
\end{split}
\end{equation*}
we conclude that $\omega_{\psi}^{n+1}|_{\mathbb F_n} = \omega_{\psi}^n$. It follows that there is a positive linear functional $\omega_{\psi}$ on $AF(\Br)$ such that $\omega_{\psi}|_{\mathbb F_n} = \omega^n_{\psi}$ for all $n$. In particular, $\left\|\omega_{\psi}\right\| = \omega_{\psi}(1) = \psi^0$. To check that $\omega_{\psi}$ is a $\beta$-KMS functional it suffices to observe that
\begin{equation*}
\begin{split}
&\omega_{\psi}^n(E^n_{\nu,\nu'}\alpha^F_{i\beta}(E^n_{\mu,\mu'})) =  
e^{-\beta(F(\mu)-F(\mu'))}\omega_{\psi}^n(E^n_{\nu,\nu'}E^n_{\mu,\mu'})) \\
& = \delta_{\mu,\nu'} \delta_{\mu',\nu} e^{-\beta(F(\mu)-F(\mu'))} e^{-\beta F(\nu)} \psi^n_{r(\nu)}  \\
& = \delta_{\mu,\nu'} \delta_{\mu',\nu}e^{-\beta F(\mu)}  \psi^n_{r(\nu)}   \ = \ \omega_{\psi}^n(E^n_{\mu,\mu'}E^n_{\nu,\nu'}) \\
\end{split}
\end{equation*}
for all $(\mu,\mu'), (\nu,\nu') \in \mathcal P_n^{(2)}$. 
It follows that we can define a map $\Phi : \varprojlim_j A^{(j)}(\beta) \to \mathbb R^+ S^{F}_{\beta}$ such that $\Phi(\psi) = \omega_{\psi}$. It is straightforward to check that $\Phi$ is the inverse to $\Psi$ and they are clearly both affine and continuous.
\end{proof}

\begin{remark}\label{28-10-18} Let $\Phi : \varprojlim_j A^{(j)}(\beta) \to \mathbb R^+ S^{F}_{\beta}$ be the affine homeomorphism of Proposition \ref{21-09-18(90)b}. If $\{\psi(n)\}$ is a sequence of elements in $\varprojlim_j A^{(j)}(\beta) $ which increases to $\psi \in \varprojlim_j A^{(j)}(\beta)$ in the sense that
$$
\psi(n)^j_v \leq \psi(n+1)^j_v  \leq \sup_k \psi(k)^j_v = \psi^j_v
$$
for all $n,j$ and $v \in \Br_j$, then $\lim_{n \to \infty} \Phi(\psi(n)) = \Phi(\psi)$ with respect to the norm in $AF(\Br)^*$ because
$$
\lim_{n \to \infty} \left\|\Phi(\psi) - \Phi(\psi(n))\right\| = \lim_{n \to \infty} \left(\psi_{v_0} - \psi(n)_{v_0}\right) = 0 \ .
$$ 
\end{remark}


\section{Constructions with Bratteli diagrams}

\subsection{Simplex cones}

We use in the following the terminology from convexity theory as developed in the book by Goodearl, \cite{G}. In particular, a convex subset $K$ of a convex cone $C$ is a base for $C$ when every non-zero element $c$ of $C$ can be written uniquely as $c = tk$ where $t \in \mathbb R^+ \backslash \{0\}$ and $k \in K$. By a \emph{simplex cone} we mean a lattice cone $C$ in a second countable locally convex real vector space which contains a compact base. Thus a lattice cone $C$ in a second countable locally convex real vector space is a simplex cone if and only if there is a continuous affine map $l: C \to \mathbb R$ such that $l(c) > 0$ for all $c \in C \backslash \{0\}$ and $\left\{c \in C: \ l(c) = 1\right\}$ is compact. Notice that a compact base in a simplex cone is a metrizable Choquet simplex by definition.

Let $A$ be a separable unital $C^*$-algebra and $\mathcal T(A)$ the real vector space of trace functionals; i.e. $\mathcal T(A)$ consists of the continuous linear functionals $\omega : A \to \mathbb C$ which satisfy that $\omega(a^*) = \overline{\omega(a)}$ and $\omega(ab) = \omega(ba)$ for all $a,b \in A$. $\mathcal T(a)$ is a second countable locally convex real vector space in the weak* topology and the cone 
$$
\mathcal T^+(A) = \left\{ \omega \in \mathcal T(A) : \ \omega(a) \geq 0 \ \forall a \geq 0 \right\}
$$
is a simplex cone by a result of Thoma, \cite{T}. It is important here that the converse is also true. 

\begin{thm}\label{24-10-18} (Blackadar) Let $C$ be a simplex cone. There is a simple unital AF algebra $A$ and an affine homeomorphism from $C$ onto $\mathcal T^+(A)$.
\end{thm}
\begin{proof} Let $K \subseteq C$ be a compact base in $C$. By Theorem 3.10 in \cite{Bl} there is a simple unital AF algebra $A$ and an affine homeomorphism from $K$ onto the tracial state space of $A$. This affine homeomorphism extends uniquely to an affine homeomorphism from $C$ onto $\mathcal T^+(A)$. 
\end{proof}

As observed in the last proof, Choquet simplexes that are affinely homeomorphic will generate simplex cones that are also affinely homeomorphic. The converse is not true in general; a simplex cone can have many bases that are distinct in the sense that they are not affinely homeomorphic:

\begin{example}\label{24-10-18a}
When $(k,m) \in \mathbb N^2$ let $\mathcal A(k,m)$ be the AF algebra
$$    
\left\{ (x_i)_{i=1}^{\infty} \in \prod_{i=1}^{\infty} M_{k+m}(\mathbb C) : \ \lim_{i \to \infty} x_i  \in M_k(\mathbb C) \oplus M_m(\mathbb C) \subseteq M_{k+m}(\mathbb C) \right\} \ .
$$ 
For any two pairs $(k,m), (k'm') \in \mathbb N^2$ the simplex cones of positive trace functionals on $\mathcal A(k,m)$ and $\mathcal A(k',m')$ are affinely homeomorphic, but the simplexes of trace states are not when 
$$
\left\{ \frac{k}{k+m}, \frac{m}{k+m}\right\} \neq \left\{\frac{k'}{k'+m'}, \frac{m'}{k'+m'}\right\} \ .
$$
\end{example}

 As the next lemma shows this phenomenon can only occur because none of the trace spaces are Bauer simplexes. Recall that a Choquet simplex is Bauer when its extreme boundary is a closed subset of the simplex.

 \begin{lemma}\label{24-10-18b} Let $C$ and $C'$ be simplex cones, $K$ a compact base in $C$ and $K'$ a compact base in $C'$. Assume that $C$ is affinely homeomorphic to $C'$.
\begin{itemize}
\item There is an affine isomorphism from $K$ onto $K'$ which restricts to a homeomorphism from the extreme boundary $\partial_e K$ of $K$ onto the extreme boundary $\partial_e K'$ of $K'$.
\item Assume that $K$ is a Bauer simplex. Then $K$ and $K'$ are affinely homeomorphic.
\end{itemize}
\end{lemma} 
\begin{proof} Let $\Phi : C \to C'$ be an affine homeomorphism and $l' : C' \to \mathbb R$ a continuous affine map such that $K' = \left\{ c \in C' : \ l'(c) = 1 \right\}$. When $k \in \partial_e K$ set $f(k) = l'\left(\Phi(k)\right)^{-1}\Phi(k)$. Then $f$ is a homeomorphism from $\partial_e K$ onto $\partial_e K'$. To extend $f$ to an affine isomorphism we use that $K$ and $K'$ are Choquet simplexes: Every element $k \in K$ is the barycenter of a unique Borel probability measure $\nu_k$ on $\partial_e K$; in symbols
$$
k = \int_{\partial_e K} \lambda \ \mathrm{d}\nu_k \ ,
$$
cf. Theorem 4.1.15 in \cite{BR}. Define $\psi(k) \in K'$ as the barycenter of the Borel probability measure $\nu_k \circ f^{-1}$ on $\partial_e K'$, i.e.
$$
\psi(k) = \int_{\partial_e K'} \lambda \ \mathrm{d} \nu_k \circ f^{-1} \ .
$$
Then $\psi : K \to K'$ is an affine isomorphism. If we assume, as we do in the second item, that $K$ is a Bauer simplex it follows that $\partial_e K'$ is compact and hence that $K'$ is also a Bauer simplex. For a Bauer simplex the barycentric decomposition of elements is an affine homeomorphism onto the Borel probability measures on the extreme boundary, see e.g. Corollary 11.20 in \cite{G}, and hence $\psi$ is a homeomorphism in this case. 
\end{proof}

The key manipulations below will be done to simplex cones, and in order to formulate the main results in terms of Choquet simplexes we will say that two compact convex sets $K$ and $K'$ are \emph{strongly affinely isomorphic} when there is an affine isomorphism $K \to K'$ whose restriction $\partial_e K \to \partial_e K'$ to the extreme boundaries is a homeomorphism. This notion is strictly stronger than affine isomorphism and strictly weaker than affine homeomorphism, but agrees with the latter for Bauer simplexes by the proof of Lemma \ref{24-10-18b}.

\subsection{Putting Bratteli diagrams together}\label{26-10-18}

The aim in this section is to prove the following

\begin{prop}\label{13-10-18c} Let $\mathbb I$ be a finite or countably infinite collection of intervals in $\mathbb R$ such that $I_0 = \mathbb R$ for at least one element $I_0 \in \mathbb I$. For each $I \in \mathbb I$ choose a simplex cone $L_I$ and for each $\beta \in \mathbb R$ set $\mathbb I_{\beta} = \left\{I \in \mathbb I: \ \beta \in I \right\}$. There is a Bratteli diagram $\Br$ and a potential $F : \Br_{Ar} \to \mathbb R$ with the following properties:
\begin{itemize}
\item For each $I \in \mathbb I$ and each $\beta \in I$ there is a closed face $L'_I$ in $\mathbb R^+ S^{F}_{\beta}$ and an affine homeomorphism of $L_I$ onto $L'_I$.
\item For each $\beta \in\mathbb R$ and each $\omega \in \mathbb R^+ S^{F}_{\beta}$ there is a unique norm-convergent decomposition
$$
\omega = \sum_{I \in \mathbb I_{\beta}} \omega_I
$$
where $\omega_I \in L'_I$.
\end{itemize}
\end{prop}

 The intervals in $\mathbb I$ can be completely arbitrary; they may be empty, closed, open, half-open or they may consist of only one number. Since we allow empty intervals we can assume that $\mathbb I$ is infinite. Let then $I_1,I_2, I_3,  \cdots $ be a numbering of the elements in $\mathbb I \backslash \{I_0\}$. For the construction of $\Br$ and $F$ we shall use the following lemma whose proof we leave to the reader.

\begin{lemma}\label{13-10-18a} For $k \geq 1$ there are sequences $\{s^k_j\}_{j=1}^{\infty}$ and $\{t^k_j\}_{j=1}^{\infty}$ of positive real numbers such that 
$$
\sum_{j=1}^{\infty} s^k_j e^{j \beta} \ + \ \sum_{j=1}^{\infty} t^k_j e^{-j \beta} \ < \ \infty
$$ 
if and only if $\beta \in I_k$.
\end{lemma}

Using Theorem \ref{24-10-18} we choose for each $k \geq 0$ a Bratteli diagram $\Br^k$ such that $L_{I_k}$ is affinely homeomorphic to $\mathcal T^+\left(AF(\Br^k)\right)$. For $k \geq 1$ we choose $\Br^k$ to be a simple Bratteli diagram such that $AF(\Br^k)$ is infinite dimensional; for $k=0$ simplicity is not needed. Let $v_0$ be the top vertex in $\Br^0$ and choose for all $j \geq 1$ a vertex $w_j \in \Br^0_j$. For $k \geq 1$ we let $v^k_0$ be the top vertex in $\Br^k$. For $k = 0,1,2, \cdots$ and $w \in \Br^k_V$ we let $\underline{\Br}(k)_w$ denote the number of paths in $\Br^k$ from the top vertex in $\Br^k$ to $w$. Since $\Br^k$ is simple when $k \geq 1$ we can use the procedure called telescoping in Definition 3.2 on page 68 of \cite{GPS} to arrange that 
\begin{equation}\label{25-10-18a}
\frac{1}{2} \frac{t^k_j \ \underline{\Br}(k)_w}{\underline{\Br}(0)_{w_{k+2j-1}}} \ \leq \ m^k_{2j}(w) \ \leq \  \frac{t^k_j \ \underline{\Br}(k)_w}{\underline{\Br}(0)_{w_{k +2j-1}}}
\end{equation}
for all $w \in \Br^k_{2j}$ when $j \geq 1$ and $m^k_{2j}(w)$ denotes the integer part of
$$
\frac{t^k_j \ \underline{\Br}(k)_w}{\underline{\Br}(0)_{w_{k+2j-1}}} \ .
$$
In the same way we arrange that 
\begin{equation}\label{25-10-18b}
\frac{1}{2} \frac{s^k_j \ \underline{\Br}(k)_w}{\underline{\Br}(0)_{w_{k+2j-2}}} \ \leq \ m^k_{2j-1}(w) \ \leq  \ \frac{s^k_j \ \underline{\Br}(k)_w}{\underline{\Br}(0)_{w_{k+2j-2}}}
\end{equation}
for all $w \in \Br^k_{2j-1}$ when $m^k_{2j-1}(w)$ denotes the integer part of
$$
\frac{s^k_j \ \underline{\Br}(k)_w}{\underline{\Br}(0)_{w_{k+2j-2}}} \ .
$$
We visualize $\Br^k$ by the following schematic sketch:
   
\begin{equation*}
\begin{xymatrix}{
\Br^k_0 \ar@{=>}[r] & \Br^k_1 \ar@{=>}[r]   & \Br^k_2 \ar@{=>}[r]   & \Br^k_3 \ar@{=>}[r]  &   \Br^k_4 \ar@{=>}[r]  & \hdots }
 \end{xymatrix}
\end{equation*}  
where 
\begin{equation*}
\begin{xymatrix}{
\Br^k_j \ar@{=>}[r] & \Br^k_{j+1} }
 \end{xymatrix}
\end{equation*}  
are the $j$'th and $j+1$'st levels in $\Br^k$ with arrows between them. The Bratteli diagram $\Br$ we will consider can now be visualized as follows.
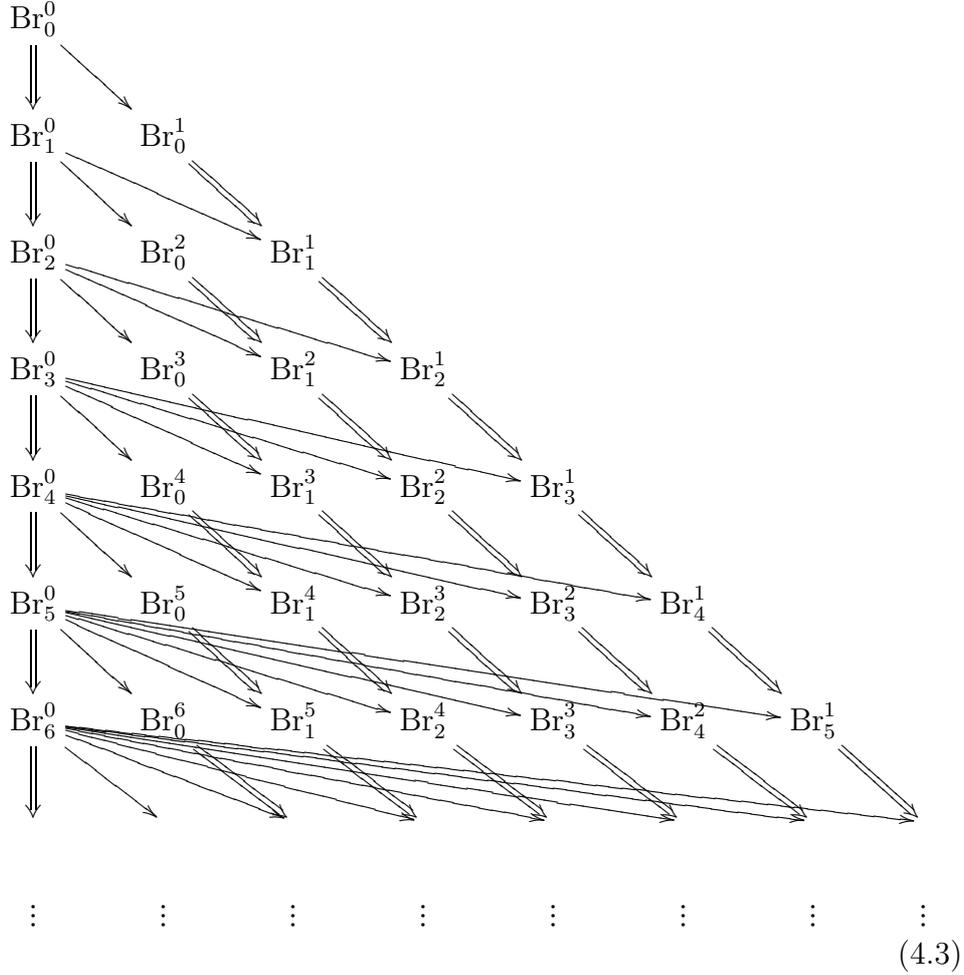
\begin{equation}\label{24-10-18n}
\begin{xymatrix}{
  \ar@{=>}@[black][d] \Br^0_0 \ar@[blue][dr]  & \\
    \ar@{=>}@[black][d] \Br^0_1 \ar@[red][rrd]  \ar@[blue][dr]  & \Br^1_0 \ar@{=>}[dr]  & \\ 
  \ar@{=>}@[black][d] \Br^0_2 \ar@[red][rrd] \ar@[blue][rrrd]  \ar@[blue][dr]  &  \ar@{=>}[dr] \Br^2_0    &  \Br^1_1 \ar@{=>}[dr]  \\ 
  \ar@{=>}@[black][d] \Br^0_3 \ar@[red][rrd] \ar@[red][rrrrd]\ar@[blue][rrrd] \ar@[blue][dr] & \Br^3_0 \ar@{=>}[dr] & \Br^2_1 \ar@{=>}[dr] &   \Br^1_2 \ar@{=>}[dr] & \\
  \ar@{=>}@[black][d] \Br^0_4 \ar@[red][rrd] \ar@[red][rrrrd]\ar@[blue][rrrd] \ar@[blue][dr] \ar@[blue][rrrrrd]  & \ar@{=>}[dr] \Br^4_0 & \ar@{=>}[dr] \Br^3_1  &    \ar@{=>}[dr] \Br^2_2&   \Br^1_3 \ar@{=>}[dr]&\\
  \ar@{=>}@[black][d] \Br^0_5  \ar@[red][rrd] \ar@[red][rrrrd]\ar@[blue][rrrd] \ar@[blue][dr] \ar@[blue][rrrrrd] \ar@[red][rrrrrrd] & \ar@{=>}[dr] \Br^5_0 & \ar@{=>}[dr] \Br^4_1  &    \ar@{=>}[dr] \Br^3_2&   \Br^2_3 \ar@{=>}[dr]&\Br^1_4 \ar@{=>}[dr] &\\ 
   \ar@{=>}@[black][d] \Br^0_6 \ar@[red][rrd] \ar@[red][rrrrd]\ar@[blue][rrrd] \ar@[blue][dr] \ar@[blue][rrrrrd] \ar@[red][rrrrrrd] \ar@[blue][rrrrrrrd] & \ar@{=>}[dr] \Br^6_0 & \ar@{=>}[dr] \Br^5_1  &    \ar@{=>}[dr] \Br^4_2&   \Br^3_3 \ar@{=>}[dr]&\Br^2_4 \ar@{=>}[dr] &\Br^1_5 \ar@{=>}[dr]\\ 
  &&&&&&&&\\
 \vdots & \vdots &\vdots& \vdots & \vdots &\vdots &\vdots & \vdots &}
 \end{xymatrix}
\end{equation}

The red and blue arrows in \eqref{24-10-18n} are quivers of arrows. For $k,j \geq 1$ the red arrow 
\begin{equation*}\label{25-10-18}
\begin{xymatrix}{
\Br^0_{k+2j-2} \ar@[red][r] & \Br^{k}_{2j-1} }
 \end{xymatrix}
\end{equation*} 
signify that for each $w \in \Br^k_{2j-1}$ there are $m^k_{2j-1}(w)$ arrows going from $w_{k+2j-2} \in \Br^0_{k+2j-2}$ to the vertex $w$, and for $j \geq 0$ the blue arrow
\begin{equation}\label{25-10-18m}
\begin{xymatrix}{
\Br^0_{k+2j-1} \ar@[blue][r] & \Br^{k}_{2j} }
 \end{xymatrix}
\end{equation} 
signify that for each $w \in \Br^k_{2j}$ there are $m^k_{2j}(w)$ arrows going from $w_{k+2j-1} \in \Br^0_{k+2j-1}$ to $w$, also for $j =0$ when we put $m^k_0(v^k_0) = 1$.

To define the potential $F : \Br_{Ar} \to \mathbb R$ we set $F(a) = 0$ for all arrows $a \in \bigcup_{k=0}^{\infty} \Br^k_{Ar} \subseteq \Br_{Ar}$. To define $F$ on the arrows in the quivers visualized by the blue and red arrows in \eqref{24-10-18n} let $\kappa : \mathbb N \cup \{0\} \to \mathbb Z$ be the function such that
$\kappa( 2j) = j \ \text{and} \ \kappa(2j+1) = - j$. We set $F(a) = \kappa(j)$ when $s(a) \in \Br^0_V$ and $r(a) \in \Br^k_j$ for some $k \geq 1$. For $k =0,1,2, \cdots $, let $\{A^{(k,j)}\}$ be the projective matrix system over $\Br^k$ corresponding to the zero potential; that is,
$$
A^{(k,j)}_{v,w} = \# r^{-1}(w) \cap s^{-1}(v)
$$
when $(v,w) \in \Br^k_{j-1} \times \Br^k_j$. Let $\beta \in \mathbb R$ and let $\{A^{(j)}(\beta)\}$ be the projective matrix system over $\Br$ given as in \eqref{27-09-18a} from the potential $F$ we have just defined.

\begin{lemma}\label{25-10-18X}
Let $\psi \in \varprojlim_j A^{(k,j)}, \ \psi \neq 0$.  
\begin{itemize}
\item There is an element $\psi' \in \varprojlim_j A^{(j)}(\beta)$ such that ${\psi'}_w= \psi_w$ for all $  w \in \Br^k_{V}$ if and only if $\beta \in I_k$, and
\item when $\beta \in I_k$ there is an element $\overline{\psi} \in \varprojlim_j A^{(j)}(\beta)$ such that 
\begin{itemize}
\item $\overline{\psi}_w= \psi_w, \ w \in \Br^k_{V}$,
\item $\overline{\psi}_v = 0$ when $v \in \Br_V \backslash \left(\Br^0_V \cup \Br^k_V\right)$, and
\item $\overline{\psi}_v \leq \varphi_v$ for all $v \in \Br_V$ when $\varphi \in \varprojlim_j A^{(j)}(\beta)$ satisfies that $\varphi_w \geq \psi_w$ for all $w \in \Br^k_V$.
\end{itemize}
\end{itemize} 
\end{lemma}
\begin{proof} The case $k =0$ is trivial and we assume that $k \geq 1$. To simplify the notation we define the matrix $A(\beta)$ over $\Br_V$ such that
$$
A(\beta)_{v,w} \ = \sum_{a \in r^{-1}(w)\cap s^{-1}(v)} e^{-\beta F(a)} \ 
$$
for all $v,w \in \Br_V$. Define $\psi^{(n)} \in [0,\infty)^{\Br_V}, \ n =0,1,2,3, \cdots$, recursively such that
$$
\psi^{(0)}_v = \begin{cases} \psi_v , \ & \ v \in \Br^k_V \\ 0, & \ v \notin \Br^k_V \ , \end{cases}
$$
and 
$$
\psi^{(n)}_v = \sum_{w \in \Br_V} A(\beta)_{v,w} \psi^{(n-1)}_w
$$
when $n \geq 1$ and $v \in \Br_V$. Then
\begin{itemize}
\item ${\psi}^{(n)}_w= \psi_w, \ w \in \Br^k_{V}$,
\item $\psi^{(n)}_v = 0$ when $v \in \Br_V \backslash \left(\Br^0_V \cup \Br^k_V\right)$,
\item $\psi^{(n-1)}_v \leq \psi^{(n)}_v, \ v \in \Br_V$, and 
\item ${\psi}^{(n)}_v \leq \varphi_v$ for all $v \in \Br_V$ when $\varphi \in \varprojlim_j A^{(j)}(\beta)$ satisfies that $\varphi_w \geq \psi_w$ for all $w \in \Br^k_V$.
\end{itemize} 
If $\lim_{n \to \infty} \psi^{(n)}_v < \infty$ for all $v \in \Br_V$ the resulting vector $\overline{\psi} = \lim_{n \to \infty} \psi^{(n)}$ will have the properties stated in the second item of the lemma. Note that this happens if and only if  $\lim_{n \to \infty} \psi^{(n)}_{v_0} \ < \ \infty$ and hence what remains is to show that $\lim_{n \to \infty} \psi^{(n)}_{v_0} \ < \ \infty$ if and only if $\beta \in I_k$. To this end we note that
$$
\psi^{(n+k-1)}_{v_0} = \sum_{j=0}^{n-1} \sum_{u \in \Br^k_{j}} \underline{\Br}(0)_{w_{k+j-1}}e^{-\beta \kappa(j)}m^k_j(u)\psi_u  \ . 
$$
It follows that $\lim_{n \to \infty} \psi^{(n)}_{v_0} \ < \ \infty$ if and only if
\begin{equation}\label{25-10-18e}
\sum_{j= 1}^{\infty} \sum_{u \in \Br^k_{2j}} \underline{\Br}(0)_{w_{k+2j-1}} e^{-\beta j} m^k_{2j}(u) \psi_u \ < \ \infty
\end{equation}
and
\begin{equation}\label{25-10-18f}
\sum_{j= 0}^{\infty} \sum_{u \in \Br^k_{2j+1}} \underline{\Br}(0)_{w_{k+2j}} e^{\beta j}m^k_{2j+1}(u) \psi_u \ < \ \infty \ .
\end{equation}
It follows from \eqref{25-10-18a} that
\begin{equation}\label{25-10-18d}
\begin{split}
&\frac{1}{2} \sum_{j= 1}^{\infty} \sum_{u \in \Br^k_{2j}}  e^{-\beta j} t^k_j \underline{\Br}(k)_u \psi_u  \ \leq \
\sum_{j= 1}^{\infty} \sum_{u \in \Br^k_{2j}} \underline{\Br}(0)_{w_{k+2j -1}} e^{-\beta j}m^k_{2j}(u) \psi_u \\
& \leq \sum_{j= 1}^{\infty} \sum_{u \in \Br^k_{2j}}  e^{-\beta j} t^k_j \underline{\Br}(k)_u \psi_u \ .
\end{split}
\end{equation}
Since $\sum_{u \in \Br^k_{2j}}  \underline{\Br}(k)_u \psi_u = \psi_{v^k_0} > 0$ it follows from \eqref{25-10-18d} that \eqref{25-10-18e} holds iff  $\sum_{j=1}^{\infty} t^k_j e^{-j \beta} \ < \ \infty$. Similarly, using \eqref{25-10-18b} it follows that \eqref{25-10-18f} holds iff  $\sum_{j=1}^{\infty} s^k_j e^{j \beta} \ < \ \infty$. By the choice of $\{s^k_j\}$ and $\{t^j_k\}$, cf. Lemma \ref{13-10-18a}, it follows that $\lim_{n \to \infty} \psi^{(n)}_{v_0} \ < \ \infty$ if and only if $\beta \in I_k$.
\end{proof}

The element $\overline{\psi}$ defined in Lemma \ref{25-10-18X} is clearly unique, being the minimal element of $\varprojlim_j A^{(j)}(\beta)$ extending $\psi$. We set $\overline{\psi} = 0$ when $\psi = 0$.

\begin{lemma}\label{23-10-18} Let $\beta \in I_k$. The map
$$
\varprojlim_j A^{(k,j)}\ni \psi \ \mapsto \ \overline{\psi} \in \varprojlim_j A^{(j)}(\beta)
$$
is injective, continuous and affine. 
\end{lemma}
\begin{proof} It suffices to show that $\psi \to \overline{\psi}_v$ is continuous and affine when $v \in \Br^0_{j-1}$ and $j-1 \geq k$. In this case, with the notation from the proof of Lemma \ref{25-10-18f},
\begin{equation}\label{25-10-18g}
\overline{\psi}_v = \sum_{n=0}^{\infty} \sum_{u \in \Br^k_{j+n-k}} A(\beta)^n_{v,w_{j+n-1}}e^{-\beta \kappa(j+n-k)}m^k_{j+n-k}(u)\psi_u \ ,
\end{equation}
when we use the convention that $A(\beta)^0$ is the identity matrix, i.e. $A(\beta)^0_{u,v} = \delta_{u,v}$. Since
$$
\psi \ \mapsto \ \sum_{n=0}^{N}  \sum_{u \in \Br^k_{j+n-k}} A(\beta)^n_{v,w_{j+n-1}}e^{-\beta \kappa(j+n-k)}m^k_{j+n-k}(u)\psi_u
$$
is continuous and affine for all $N$, it suffices to show that the sum \eqref{25-10-18g} converges uniformly on 
$$
\left\{ \psi \in \varprojlim_j A^{(k,j)}: \  \psi_{v^k_0} \leq R \right\}
$$
for each $R > 0$. For this consider first an $n$ such that $j+n-k$ is even; say $j+n-k = 2l$. Since
$$
 A(\beta)^n_{v,w_{j+n-1}} \leq \underline{\Br}(0)_{w_{j+n-1}}
$$
it follows from \eqref{25-10-18a} that
\begin{equation*}
\begin{split}
&\sum_{u \in \Br^k_{j+n -k}} A(\beta)^n_{v,w_{j+n-1}}e^{-\beta \kappa(j+n-k)}m^k_{j+n-k}(u)\psi_u \\
& \leq \sum_{u \in \Br^k_{j+n-k}} t^k_{l} e^{-\beta l} \underline{\Br}(k)_u \psi_u \ = \ t^k_{l} e^{-\beta l} \psi_{v^k_0} \leq Rt^k_{l} e^{-\beta l} \ 
\end{split}
\end{equation*}
when $\psi_{v^k_0} \leq R$. When $j+n-k$ is odd, say $j+n -k= 2l+1$, we find in the same way that
$$
\sum_{u \in \Br^k_{j+n-k}} A(\beta)^n_{v,w_{j +n-1}}e^{-\beta \kappa(j+n-k)}m^k_{j+n-k}(u)\psi_u \leq Rs^k_{l} e^{\beta l} \ 
$$
when $\psi_{v^k_0} \leq R$. Therefore the desired uniform convergence follows because $\beta \in I_k$.
\end{proof}

\begin{lemma}\label{23-10-18a} Let $\beta \in I_k$. The set $\left\{ \overline{\psi} : \ \psi \in \varprojlim_j A^{(k,j)}\right\}$ is a closed convex face in $\varprojlim_j A^{(j)}(\beta)$.
\end{lemma}
\begin{proof} Set $\mathcal F_k = \left\{ \overline{\psi} : \ \psi \in \varprojlim_j A^{(k,j)}\right\}$. It follows from Lemma \ref{23-10-18} that $\mathcal F_k$ is closed and convex. To show that $\mathcal F_k$ is a face in $\varprojlim_j A^{(j)}(\beta)$ consider $\psi \in \varprojlim_j A^{(k,j)}, \ \varphi^1,\varphi^2 \in\varprojlim_j A^{(j)}(\beta)$ and $t \in ]0,1[$ such that $\overline{\psi} = t\varphi^1 + (1-t)\varphi^2$. It follows from the minimality property in Lemma \ref{25-10-18X} that 
$$
\varphi^i \geq \overline{\varphi^i|_{\Br^k_V}}, \ i = 1,2 \ .
$$
Since $\psi = t\varphi^1|_{\Br^k_V} +(1-t)\varphi^2|_{\Br^k_V}$ we have also that 
$$
 \overline{\psi} = t\overline{\varphi^1|_{\Br^k_V}} + (1-t)\overline{\varphi^2|_{\Br^k_V}} \ ,
 $$
 and hence that
 $$
 \overline{\psi}  = t\varphi^1 + (1-t)\varphi^2 \geq t \overline{\varphi^1|_{\Br^k_V}} + (1-t) \overline{\varphi^2|_{\Br^k_V}} = \overline{\psi} \ .
 $$
 It follows that $\varphi^i =\overline{\varphi^i|_{\Br^k_V}} \in \mathcal F_k, \ i = 1,2$.
\end{proof}

\emph{Proof of Proposition \ref{13-10-18c}:} We use Proposition \ref{21-09-18(90)b} to identify $\mathbb R^+ S^{F}_{\beta}$ and $\varprojlim_j A^{(j)}(\beta)$, and in that picture we set
$$
L'_I = \left\{\overline{\psi} : \ \psi \in \varprojlim_j A^{(k,j)} \right\} \ 
$$
when $I = I_k$. In this way the statement in the first item in Proposition \ref{13-10-18c} follows from the preceding lemmas.  To prove the statement of the second item set $\mathbb J_{\beta} = \left\{k \in \mathbb N \cup \{0\}: \ \beta \in I_k \right\}$ and let $\psi \in  \varprojlim_j A^{(j)}(\beta)$. It follows from the first item in Lemma \ref{25-10-18X} that $\psi|_{\Br^m_V} = 0$ when $m \notin \mathbb J_{\beta}$ and from the second that
$$
\sum_{k \in \mathbb J_{\beta} \backslash \{0\}} \overline{\psi|_{\Br^k_V}} \leq \psi \ .
$$ 
Set
$$
\psi^0 = \psi \ - \sum_{k \in \mathbb J_{\beta} \backslash \{0\}} \overline{\psi|_{\Br^k_V}} 
$$
and note that $\psi^0 \in \mathcal F_0$. Set $\psi^k = \overline{\psi|_{\Br^k_V}}, \ k \in \mathbb J_{\beta}\backslash \{0\}$, and note that
$$
\psi = \sum_{k \in \mathbb J_{\beta}} \psi^k
$$
with point-wise convergence on $\Br_V$. It follows from Remark \ref{28-10-18} that the corresponding sum in $AF(\Br)^*$ is norm-convergent.
\qed

\begin{prop}\label{13-11-18} In the setting of Proposition \ref{13-10-18c}, assume that $L_I$ has a Bauer simplex $S_I$ as a base for all $I \in \mathbb I$ and that $S_{I_0}$ only contains one element. It follows that $S^F_{\beta}$ is a Bauer simplex whose extreme boundary $\partial_e S^F_{\beta}$ is homeomorphic to the one-point compactification of the topological disjoint union
\begin{equation}\label{13-11-18a}
\bigsqcup_{I \in \mathbb I_{\beta} \backslash \{I_0\}} \partial_e S_I \ 
\end{equation}
for all $\beta \in \mathbb R$.
\end{prop}
\begin{proof} As in the preceding proof we identify $\varprojlim_j A^{(j)}(\beta)$ with $\mathbb R^+ S^F_{\beta}$. Let $*$ be the point at infinity in the one-point compactification of \eqref{13-11-18a}. By Proposition \ref{13-10-18c} we have a continuous injective map
$$
\Phi : \bigsqcup_{I \in \mathbb I_{\beta} \backslash \{I_0\}} \partial_e S_I   \ \to \ \partial_e S^F_{\beta} 
$$
which we extend such that $\Phi(*) = \omega_0$ where $\omega_0$ is the unique element of $L'_{I_0} \cap S^F_{\beta}$. It follows from Proposition \ref{13-10-18c} that the extension is both injective and surjective. It suffices therefore to show that it is also continuous. Let $\{x_n\}$ be a sequence in \eqref{13-11-18a} converging to $*$. Set $\psi^n = \Phi(x_n)$ and note that $\{\psi^n\}$ is a sequence in $\partial_e S^F_{\beta}$ which eventually leaves $L'_I$ for all $I \in \mathbb I_{\beta} \backslash \{I_0\}$. It follows that every condensation point $\psi$ of $\{\psi^n\}$ must satisfy that $\psi_{v^k_0} = 0$ for all $k \geq 1$ and hence that $\psi \in L'_{I_0} \cap S^F_{\beta} = \{\omega_0\}$. It follows that $\lim_{n \to \infty} \psi^n = \omega_0$.

\end{proof}

\section{Generalized gauge actions on a UHF algebra} 

Given a Bratteli diagram $\Br$ we define $\Br_{j} \times \Br_{j-1}$-matrices $\Br^{(j)} = \left(\Br^{(j)}_{v,w}\right)$ such that
$$
\Br^{(j)}_{v,w} = \# r^{-1}(v) \cap s^{-1}(w) \ ,
$$
i.e. $\Br^{(j)}_{v,w}$ is the number of arrows from $w \in \Br_{j-1}$ to $v \in \Br_j$. Note that $\Br^{(j)}$ is the transpose of the matrix $A^j(\beta)$ in the projective matrix system over $\Br$ corresponding to the zero potential, cf. \eqref{27-09-18a}. We call $\{\Br^{(j)}\}$ the multiplicity matrices of $\Br$. 

We fix in the following an (infinite dimensional) UHF algebra $U$, given by a Bratteli diagram with one vertex on each level and multiplicity matrices given by natural numbers $d_j \geq 2$.

\begin{lemma}\label{01-11-18} Let $\Br$ be a Bratteli diagram. There is a Bratteli diagram $\Br'$ such that $\Br_V = \Br'_V$, $\Br'^{(j)}_{v,w} > \Br^{(j)}_{v,w}$ for all $(v,w) \in \Br_j \times \Br_{j-1}$ and
$$
\phi\left(\Br'^{(j)} - \Br^{(j)}\right) \geq \frac{1}{4}
$$ 
for all $j = 1,2,3, \cdots$, and $AF(\Br')$ is $*$-isomorphic to $U$.
\end{lemma}
\begin{proof} For $j =1,2,3, \cdots$, let $C_j$ be a set consisting of one element $v_j$, and let $\Br''$ be the Bratteli diagram with level sets
$\Br''_{2j-1} = \Br_j, j =1,2,\cdots$, $\Br''_0 = \Br_0$ and $\Br''_{2j} = C_j, \ j = 1,2,3, \cdots$. Choose natural numbers $0 = k_0 < k_1 < k_2 < \cdots$ such that when we write
$$
\prod_{i=k_{j-1} +1}^{k_j} d_i  = S_j\left(\# \Br_j\right) + r_j \ ,
$$
where $S_j,r_j \in \mathbb N$, $r_j \leq \#\Br_j$, the estimate
\begin{equation}\label{01-11-18c}
\max \left\{ \frac{ \Br^{(j)}_{v,w}}{S_j} : \ (v,w) \in \Br_j \times \Br_{j-1} \right\} \leq \frac{1}{2} \ 
\end{equation}
holds. Choose an element $u_j \in \Br_j$ for all $j \geq 1$. For $j = 1,2,3, \cdots$ set
\begin{equation*}\label{26-09-18}
\Br''^{(2j-1)}_{v,v_{j-1}} =  \begin{cases} S_j+r_j , & \ v = u_{j} ,\\  S_j , & \ v \in \Br_j \backslash \{ u_{j}\} \  \end{cases}
\end{equation*}
and
$$
\Br''^{(2j)}_{v_j,v} = 1
$$
for all $v \in \Br''_{2j-1}$. The matrices $\{\Br''^{(j)}\}$ are the multiplicity matrices of $\Br''$. Let $\Br'$ be the Bratteli diagram obtained by removing from $\Br''$ the even level sets $\Br''_{2j} = C_j, \ j = 1,2,\cdots$, and telescoping as explained in Definition 3.2 on page 68 in \cite{GPS}. Then $\Br'_V = \Br_V$ and it follows from the choices made above that $\Br'^{(j)}$ has the stated properties: Since $\Br'^{(j)}_{x,y} = \Br'^{(j)}_{x,y'}$ it follows from the definition of $\phi$ that
$$
\phi\left(\Br'^{(j)}\right) \ \geq \ 1,
$$
and then from \eqref{01-11-18c} and Lemma \ref{01-11-18b} that
$$
\phi\left(\Br'^{(j)} - \Br^{(j)}\right) \ \geq  \ \frac{1}{4} \phi\left(\Br'^{(j)}\right) \ \geq \ \frac{1}{4} \ . 
$$
If we instead remove the odd level sets $\Br''_{2j-1} = \Br_j, \ j \geq 1$, in $\Br''$ and telescope, we obtain a Bratteli diagram for $U$. It follows that $AF(\Br') \simeq AF(\Br'') \simeq U$.
\end{proof}

The norms we consider in the following are the Euclidean norms (or $l^2$-norms) on vectors, and on matrices it is the corresponding operator norm.
\begin{lemma}\label{19-09-18b}  Let $\Br$ be a Bratteli diagram and $\{A^{(j)}\}$ a projective matrix system over $\Br$. Set 
$$
L_j = \sqrt{\# \Br_j} \max_{w \in \Br_j} \left(\left( A^{(1)}A^{(2)} \cdots A^{(j)}\right)_{v_0,w}\right)^{-1} \ 
$$
and choose $ 0 < \epsilon_k < 1, \ k =1,2,3, \cdots$, such that for some $N \in \mathbb N$ we have that
\begin{equation}\label{firstepsb}
\epsilon_k \leq 2^{-k} \left(L_k \prod_{j=1}^{k-1} \left( \left\|A^{(j)}\right\| + 1\right)\right)^{-1} \ 
\end{equation}
for all $k \geq N$. Let $\{B^{(j)}\}$ be a projective matrix system over $\Br$ such that 
\begin{itemize}
\item $\left(B^{(1)}B^{(2)} \cdots B^{(j)}\right)_{v_0,w} \geq \left(A^{(1)}A^{(2)} \cdots A^{(j)}\right)_{v_0,w}$ for all $w \in \Br_j$ and all $j \in \mathbb N$, and
\item $\left\|A^{(j)}-B^{(j)} \right\| \leq \epsilon_j $ for all $j \geq N$.
\end{itemize}
There is an affine homeomorphism $ T : \varprojlim_j A^{(j)} \to \varprojlim_j B^{(j)}$ such that 
 $$
 (T\psi)^{j-1} = \lim_{k \to \infty} B^{(j)}B^{(j+1)} \cdots B^{(j+k)}\psi^{j+k}
 $$ 
 when $\psi = \left(\psi^j\right)_{j=0}^{\infty} \in \varprojlim_j A^{(j)}$.
\end{lemma}
\begin{proof} Let $\psi = (\psi^j)_{j=0}^{\infty} \in \varprojlim_j A^{(j)}$. It follows from 1) in Lemma \ref{21-09-18bb} that $\left\|\psi^j\right\| \leq L_j \psi^0$. Let 
$M \geq \max \left\{ \left\|B^{(j)}\right\|, \ \left\|A^{(j)}\right\|, \ 1 \right\}$
for all $j = 1,2,3, \cdots , N$. Then
\begin{equation}\label{20-09-18db}
\begin{split}
& \left\|B^{(j)}B^{(j+1)} \cdots B^{(j+k)}\psi^{j+k} - B^{(j)}B^{(j+1)} \cdots B^{(j+k+1)}\psi^{j+k+1} \right\| \\
& = \left\|B^{(j)}B^{(j+1)} \cdots B^{(j+k)}A^{(j+k+1)}\psi^{j+k+1} - B^{(j)}B^{(j+1)} \cdots B^{(j+k+1)}\psi^{j+k+1} \right\| \\
& \leq \left\| B^{(j)}B^{(j+1)} \cdots B^{(j+k)}\right\| \left\| A^{(j+k+1)}\psi^{j+k+1}  - B^{(j+k+1)}\psi^{j+k+1}\right\| \\
& \leq M^N \prod_{n=N+1}^{j+k} \left(\|A^{(n)}\|+1\right) \epsilon_{j+k+1} L_{j+k+1}\psi^0 \ \leq \ M^N 2^{-j-k-1} \psi^0
\end{split}
\end{equation}
for all $k \geq N$ and  all $j \in  \mathbb N$. It follows that
$$
\left\{ B^{(j)}B^{(j+1)} \cdots B^{(j+k)}\psi^{j+k}\right\}_{k=1}^{\infty}
$$
is a Cauchy-sequence in $\mathbb R^{\Br_{j-1}}$ and we set
$$
\phi^{j-1} = \lim_{k \to \infty} B^{(j)}B^{(j+1)} \cdots B^{(j+k)}\psi^{j+k} \ ,
$$
$j = 1,2,3, \cdots $. Note that $\phi = (\phi^j)_{j=0}^{\infty} \in \varprojlim_j B^{(j)}$. The assignment $T \psi = \phi$ is a map $T : \varprojlim_j A^{(j)} \to \varprojlim_j B^{(j)}$ which is clearly affine. It follows from \eqref{20-09-18db} that
$$
\left\|\phi^{j-1} -  B^{(j)}B^{(j+1)} \cdots B^{(j+k)}\psi^{j+k}\right\| \leq M^N \psi^0\sum_{n=j+k+1}^{\infty} 2^{-n} \ 
$$
for all $j \geq 1$ and $k \geq N$, implying that $T$ is continuous. 

To construct an inverse map, let $\phi = \left(\phi^j\right)_{j=0}^{\infty} \in \varprojlim_j B^{(j)}$. It follows from the first condition in the lemma that
$\phi^j_w \leq   \left(\left(A^{(1)}A^{(2)} \cdots A^{(j)}\right)_{v_0,w}\right)^{-1}\phi^0$ for all $w \in \Br_j$, and hence that $\left\|\phi^j\right\| \leq L_j\phi^0$. Therefore an estimate analogues to \eqref{20-09-18db} shows that
$$
 \left\|A^{(j)}A^{(j+1)} \cdots A^{(j+k)}\phi^{j+k} - A^{(j)}A^{(j+1)} \cdots A^{(j+k+1)}\phi^{j+k+1} \right\| \leq  M^N 2^{-j-k-1}\phi^0
 $$
 for $j \in \mathbb N$ and $k \geq N$, and we can therefore define 
 $$
 (S\phi)^{j-1} = \lim_{k \to \infty} A^{(j)}A^{(j+1)} \cdots A^{(j+k)}\phi^{j+k} 
 $$
for all $j = 0,1,2,\cdots$, giving us a continuous affine map $S : \varprojlim_j B^{(j)} \to \varprojlim_j A^{(j)}$. Let $\psi \in \varprojlim_j A^{(j)}$. Note that
\begin{equation*}
\begin{split}
&\left\| A^{(j)}\cdots A^{(j+k)}B^{(j+k+1)}\cdots B^{(j+k+m)}\psi^{j + k +m} - \psi^{j-1} \right\| \\
& = \left\| A^{(j)}\cdots A^{(j+k)}B^{(j+k+1)}\cdots B^{(j+k+m)}\psi^{j + k +m} - A^{(j)}A^{(j+1)}\cdots A^{(j+k+m)}\psi^{j + k +m} \right\| \\
&\leq \prod_{n=j}^{j+k} \left\|A^{(n)}\right\|\left\| B^{(j+k+1)}\cdots B^{(j+k+m)}\psi^{j + k +m} - \psi^{j+k} \right\| \\
& \leq \prod_{n=j}^{j+k} \left\|A^{(n)}\right\|\prod_{n=j+k+1}^{j+k+m-1} \left(\left\|A^{(n)}\right\| +1\right) \epsilon_{j+k+m}L_{j+k+m}\psi^0 \\
& \ \ \ \ \ \ \ \ \ + \prod_{n=j}^{j+k} \left\|A^{(n)}\right\|\left\| B^{(j+k+1)}\cdots B^{(j+k+m-1)}\psi^{j + k +m-1} - \psi^{j+k} \right\| \\
& \leq  M^N 2^{-j-k-m} \psi^0 + \prod_{n=j}^{j+k} \left\|A^{(n)}\right\|\left\| B^{(j+k+1)}\cdots B^{(j+k+m-1)}\psi^{j + k +m-1} - \psi^{j+k} \right\|\\
& \leq \ \cdots \\
& \leq  M^N \sum_{n=j+k+1}^{j+k+m} 2^{-n} \psi^0\ 
\end{split}
\end{equation*}
when $j,m \in \mathbb N$ and $k \geq N$. Letting $m$ go to infinity we find that
\begin{equation*}
\begin{split}
&\left\| A^{(j)}\cdots A^{(j+k)}(T\psi)^{j+k}- \psi^{j-1} \right\| \leq M^N  2^{-j-k} \psi^0
\end{split}
\end{equation*}
for all $j \in \mathbb N$, $k \geq N$. Letting $k$ go to infinity it follows that $\left(S\left(T\psi\right)\right)^{j-1} = \psi^{j-1}$, proving that $S \circ T$ is the identity on $\varprojlim_j A^{(j)}$. It follows in the same way that $T \circ S$ is the identity map on $\varprojlim_j B^{(j)}$.
\end{proof}


\begin{lemma}\label{23-09-18f} Let $\Br$ be a Bratteli diagram and $\alpha^F$ a generalized gauge action on $AF(\Br)$. There is a generalized gauge action $\alpha$ on $U$ such that the convex cones $\mathbb R^+ S^{F}_{\beta}$ and $\mathbb R^+ S^{\alpha}_{\beta}$ of positive $\beta$-KMS functionals for $\alpha^F$ and $\alpha$ are affinely homeomorphic when $\beta >0$ and $\mathbb R^+S^{\alpha}_{\beta} \simeq \mathbb R^+$ when $\beta < 0$. 
\end{lemma}
\begin{proof} By Lemma \ref{01-11-18} there is a Bratteli diagram $\Br'$ such that $\Br'_V = \Br_V$ and $\Br'^{(j)} >  \Br^{(j)}$ entry wise,
$$
\phi\left(\Br'^{(j)} - \Br^{(j)} \right) \geq \frac{1}{4} 
$$ 
and $AF(\Br') \simeq U$. If we let $\mathcal A^{(j)}_{v,w}$ and $\mathcal A'^{(j)}_{v,w}$ denote the set of arrows from $w \in \Br_{j-1}$ to $v \in \Br_j$ in $\Br$ and $\Br'$, respectively, we have that 
$$
\mathcal A^{(j)}_{v,w} \subseteq \mathcal A'^{(j)}_{v,w} \ 
$$
since $\Br'^{(j)} \geq \Br^{(j)}$. For each $\beta \in \mathbb R$ we consider the projective matrix system $\{A^{(j)}(\beta)\}$ from \eqref{27-09-18a} corresponding to the given potential $F$. For each $k \in \mathbb N$ we choose $0 <\epsilon_k < 1$ such that
$$
\epsilon_k \leq  2^{-k} \left(L(\beta)_k \prod_{j=1}^{k-1} \left( \left\|A^{(j)}(\beta)\right\| + 1\right)\right)^{-1} \ 
$$
for all $\beta \in  [k^{-1},k]$, where
$$
L(\beta)_k =\sqrt{\# \Br_k} \max_{w \in \Br_k} \left(\left( A^{(1)}(\beta)A^{(2)}(\beta) \cdots A^{(k)}(\beta)\right)_{v_0,w}\right)^{-1} \ .
$$
 Let $t_k$ be a positive real number and extend $F$ to $F' : \Br'_{Ar} \to \mathbb R$ such that
$F'(a) = t_k$ when $a \in \mathcal A'^{(k)}_{v,w} \backslash \mathcal A^{(k)}_{v,w}$, and let 
$$
A'^{(k)}(\beta) = \left(A'^{(k)}(\beta)_{w,v}\right)_{(w,v) \in \Br_{k-1} \times \Br_k} 
$$
be the resulting projective matrix system over $\Br'$, i.e.
$$
A'^{(k)}(\beta)_{w,v} \ = \sum_{a \in \mathcal A'^{(k)}_{v,w}} e^{-\beta F'(a)} \ = \ A^{(k)}(\beta)_{w,v} \ \ + \ \# \left(\mathcal A'^{(k)}_{v,w} \backslash \mathcal A^{(k)}_{v,w}\right)  \ e^{-\beta t_k} \ .
$$
By choosing $t_k$ large enough we arrange that
$$
\left\|A^{(k)}(\beta) - A'^{(k)}(\beta)\right\| \ \leq \ \epsilon_k \ 
$$
for all $\beta \in  [k^{-1},k]$. Note that the matrix
$$
B(k) = \left(\# \left(\mathcal A'^{(k)}_{v,w} \backslash \mathcal A^{(k)}_{v,w}\right)\right)_{(w,v) \in \Br_{k-1}\times \Br_k}
$$
is strictly positive and the transpose of $\Br'^{(k)} -\Br^{(k)}$. By choosing $t_k$ large enough we arrange that
$$
\frac{A^{(k)}(\beta)_{w,v}}{e^{-\beta t_k}  B(k)_{w,v}} \leq \frac{1}{2}
$$
for all $(w,v) \in \Br_{k-1}\times \Br_k$ when $\beta \in [-k,-k^{-1}]$. Using Lemma \ref{01-11-18b} this leads to
\begin{equation*}
\begin{split}
& \phi\left(A'^{(k)}(\beta)\right) \ \geq \ \frac{4}{9} \phi\left( e^{-\beta t_k}B(k)\right) \ = \  \frac{4}{9} \phi\left( B(k)\right) \\
& = \
 \frac{4}{9} \phi\left( \Br'^{(k)} -\Br^{(k)}\right) \ \geq \ \frac{1}{9} \ 
\end{split}
\end{equation*}
for all $\beta \in [-k,-k^{-1}]$. It follows then from Lemma \ref{19-09-18b} that $\varprojlim_j A^{(j)}(\beta)$ is affinely homeomorphic to $\varprojlim_j A'^{(j)}(\beta)$ for all $\beta  >  0$ and from Lemma \ref{30-10-18} that $\varprojlim_j A'^{(j)}(\beta) \simeq \mathbb R^+$ for all $\beta < 0$. The potential $F'$ defines a generalized gauge action on $AF(\Br')$ and it follows from Proposition \ref{21-09-18(90)b} that $\varprojlim_j A'^{(j)}(\beta)$ is affinely homeomorphic to the cone $\mathbb R^+ S^{F'}_{\beta}$ while $\varprojlim_j A^{(j)}(\beta)$ is affinely homeomorphic to the cone $\mathbb R^+ S^{F}_{\beta}$. The isomorphism $AF(\Br') \simeq U$ provides us then with a generalized gauge action $\alpha$ on $U$ with the stated property.
\end{proof}

\begin{lemma}\label{31-10-18} Let $\Br$ and $\Br'$ be Bratteli diagrams with potentials $F:\Br_{Ar} \to \mathbb R$ and $F' : \Br'_{Ar} \to \mathbb R$.
\begin{itemize}
\item The automorphism $\alpha$ group on $AF(\Br) \otimes AF(\Br')$ defined such that $\alpha_t = \alpha^F_t \otimes \alpha^{F'}_{-t}$ is a generalized gauge action.
\item Let $\beta \in \mathbb R$ and assume that there is a unique $(-\beta)$-KMS state $\omega'$ for $\alpha^{F'}$. It follows that the map 
$$
S^F_{\beta}\ni \omega \ \mapsto \ \omega \otimes \omega'
$$
is an affine homeomorphism from $S^F_{\beta}$ onto $S^{\alpha}_{\beta}$.
\end{itemize}
\end{lemma}
\begin{proof} $AF(\Br) \otimes AF(\Br')$ is $*$-isomorphic to $AF(\Br \times \Br')$ where 
$$
\left(\Br\times \Br'\right)_j = \Br_j \times \Br'_j
$$ 
with multiplicity matrices given by $\left(\Br \times \Br'\right)^{(j)}_{(v,w), (v',w')} = \Br^{(j)}_{v,v'}\Br'^{(j)}_{w,w'}$. Under this $*$-isomorphism $\alpha$ corresponds to the generalized gauge action defined by a potential $F''$ such 
$$F''(a,a')= F(a) - F'(a')$$ 
when $(a,a') \in \left(\Br \times \Br'\right)_{Ar} \subseteq \Br_{Ar} \times \Br'_{Ar}$. To prove the second item it suffices to show that the map in the statement is surjective. Let $\left\{E^n_{\mu,\mu'}\right\}$ and $\left\{E'^n_{\nu,\nu'}\right\}$ be the matrix units from \eqref{22-09-18} in $AF(\Br)$ and $AF(\Br')$, respectively. Let $P : AF(\Br) \to D$ be the canonical conditional expectation onto the diagonal $D \subseteq AF(\Br)$. When $\omega$ is a $\beta$-KMS state for $\alpha$ we find that
\begin{equation*}
\begin{split}
&\omega\left( E^n_{\mu,\mu'} \otimes E^n_{\nu,\nu'}\right) = \delta_{\mu,\mu'} \delta_{\nu,\nu'} \omega\left(E^n_{\mu, \mu} \otimes E^n_{\nu,\nu}\right) = \omega \left( P(E^n_{\mu, \mu'}) \otimes E'^n_{\nu, \nu'}\right) \ ,
\end{split}
\end{equation*}
showing that $\omega$ factorises through $P \otimes \id_{AF(\Br')}$. Let $a \in AF(\Br)$ be a positive element fixed by $\alpha^F$. Then $b \mapsto \omega(a  \otimes b)$ is a positive $(-\beta)$-KMS functional for $\alpha^{F'}$ and hence  
$$
\omega( a \otimes b) = \lambda(a) \omega'(b) \ \ \forall b \in AF(\Br')
$$
for some $\lambda(a) \geq 0$. Since the diagonal $D \subseteq AF(\Br)$ is in the fixed point algebra of $\alpha^F$ and $\omega$ factorises through $P \otimes \id_{AF(\Br')}$ it follows that
$\omega = \omega'' \otimes \omega'$ where $\omega'' \in S^F_{\beta}$.

\end{proof}

\begin{thm}\label{25-09-18f} Let $\Br$ be a Bratteli diagram and $\alpha^F$ a generalized gauge action on $AF(\Br)$. There is a generalized gauge action $\alpha$ on the UHF algebra $U$ such that for all $\beta \neq 0$ the simplex cone $\mathbb R^+S^{\alpha}_{\beta}$ is affinely homeomorphic to $\mathbb R^+S^F_{\beta}$.
\end{thm}
\begin{proof}  Write $U = U^+ \otimes U^-$ where $U^{\pm}$ are both UHF algebras. It follows from Lemma \ref{23-09-18f} that there are generalized gauge actions $\alpha^{\pm}$ on $U^{\pm}$ such that $\mathbb R^+ S^{\alpha^{\pm}}_{\beta}$ is affinely homeomorphic to $\mathbb R^+ S^{\pm F}_{\beta}$ when $\beta > 0$ while $S^{\alpha^{\pm}}_{\beta}$ only contains one element when $\beta \leq 0$. It follows then from Lemma \ref{31-10-18} that $\alpha_t = \alpha^+_t \otimes \alpha^-_{-t}$ is a generalized gauge action on $U$ with the stated properties.
\end{proof}

By combining Theorem \ref{25-09-18f} with Proposition \ref{13-10-18c} we obtain the following result.

\begin{thm}\label{CAR1} Let $U$ be a UHF algebra and let $\mathbb I$ be a finite or countably infinite collection of intervals in $\mathbb R$ such that $I_0 = \mathbb R$ for at least one element $I_0 \in \mathbb I$. For each $I \in \mathbb I$ choose a simplex cone $L_I$ and for $\beta \in \mathbb R$ set  $\mathbb I_{\beta} = \{I \in \mathbb I: \ \beta \in I \}$. There is a generalized gauge action $\alpha$ on $U$ with the following properties.
\begin{itemize}
\item For each $I \in \mathbb I$ and each $\beta \in I\backslash \{0\}$ there is a closed face $L'_I$ in $\mathbb R^+S^{\alpha}_{\beta}$ and an affine homeomorphism $L_I \simeq L'_I$.
\item For each $\beta \neq 0 $ and each positive $\beta$-KMS functional $\omega \in \mathbb R^+S^{\alpha}_{\beta}$ there is a unique norm-convergent decomposition
$$
\omega = \sum_{I \in \mathbb I_{\beta}} \omega_I \ 
$$
where $\omega_I \in L'_I$.
\end{itemize}
\end{thm}

\begin{cor}\label{25-10-18k} Let $U$ be a UHF algebra and let $\mathbb I$ be a finite or countably infinite collection of intervals in $\mathbb R$ such that $I_0 = \mathbb R$ for at least one element $I_0 \in \mathbb I$. For each $I \in \mathbb I$ choose a metrizable Choquet simplex $S_I$ and for $\beta \in \mathbb R$ set  $\mathbb I_{\beta} = \{I \in \mathbb I: \ \beta \in I \}$. There is a generalized gauge action $\alpha$ on $U$ with the following properties.
\begin{itemize}
\item For each $I \in \mathbb I$  and each $\beta \in I\backslash \{0\}$ there is a closed face $F_I$ in $S^{\alpha}_{\beta}$ which is strongly affinely isomorphic to $S_I$. 
\item For each $\beta \neq 0$ and each $\beta$-KMS state $\omega \in S^{\alpha}_{\beta}$ there is a unique norm-convergent decomposition
$$
\omega = \sum_{I \in \mathbb I_{\beta}} \omega_I \ 
$$
where $\omega_I \in \mathbb R^+F_I$. 
\end{itemize}
\end{cor}
\begin{proof} Realize $S_I$ as a base of a simplex cone $L_I$ and apply Theorem \ref{CAR1}. Set $F_I = L'_I \cap S^{\alpha}_{\beta}$. Then $F_I$ is a base in $L'_I$ and a closed face in $S^{\alpha}_{\beta}$, and $F_I$ is strongly affinely isomorphic to $S_I$ by Lemma \ref{24-10-18b}.
\end{proof}

\begin{remark}\label{27-10-18a} In the setting of Corollary \ref{25-10-18k} the extreme boundary $\partial_e S^{\alpha}_{\beta}$ of $S^{\alpha}_{\beta}$ can be identified  as a set with the disjoint union
\begin{equation}\label{27-10-18}
\bigsqcup_{I \in \mathbb I_{\beta}} \partial_e S_I \ .
\end{equation}
As a topological space, however, the extreme boundary $\partial_e S_{\beta}^{\alpha}$ will generally not be equal to the topological disjoint union \eqref{27-10-18}. Concerning the topology of $\partial_e S_{\beta}^{\alpha}$ we make the following observations.
\begin{itemize}
\item The set $\partial_e S_I$ is closed in $\partial_e S_{\beta}^{\alpha}$ for all $I \in \mathbb I_{\beta}$ and open for all $I \in \mathbb I_{\beta} \backslash \{I_0\}$, but $\partial_e S_{I_0}$ is not necessarily open in $\partial_e S_{\beta}^{\alpha}$ when $\mathbb I_{\beta}$ is infinite.
\item  When $\mathbb I_{\beta}$ is a finite set the topological space $\partial_e S_{\beta}^{\alpha}$ will be homeomorphic to the topological disjoint union \eqref{27-10-18}.
\item When $\mathbb I_{\beta}$ is infinite, $S_{I_0}$ only contains one point and $S_I$ is a Bauer simplex for all $I \in \mathbb I_{\beta}$ the simplex $S^{\alpha}_{\beta}$ is Bauer and $\partial_e S_{\beta}^{\alpha}$ is homeomorphic to the one-point compactification of the topological disjoint union
\begin{equation}\label{30-10-18b}
\bigsqcup_{I \in \mathbb I_{\beta} \backslash \{I_0\}} \partial_e S_I \ .
\end{equation}
The element in $S_{I_0}$ will be the point at infinity. See Proposition \ref{13-11-18}.
\end{itemize}

\end{remark}

\section{Extreme variation of KMS simplexes} 

In this section we prove Theorem \ref{CAR2} from the introduction and we combine it with the examples of Bratteli, Elliott and Herman from \cite{BEH}. The methods are identical to those used in \cite{Th2}. 

\smallskip

\emph{Proof of Theorem \ref{CAR2}:}
  Take $\mathbb I$ to be a countable collection of intervals in $\mathbb R$ such that 
\begin{itemize}
\item $\mathbb I$ contains the interval $\mathbb R$, 
\item $\mathbb I_{\beta} \neq \mathbb I_{\beta'}$ when $\beta \neq \beta'$, where $\mathbb I_{\beta} = \left\{I \in \mathbb I: \ \beta \in I \right\}$, and
\item $\# \mathbb I_{\beta} \geq 2$ when $\beta \neq 0$.
\end{itemize}
For example, $\mathbb I$ could consist of $\mathbb R$ and all bounded intervals with rational endpoints. Choose an interval $I_0 \in \mathbb I$ such that $I_0 = \mathbb R$, and for each $I \in \mathbb I \backslash \{I_0\}$ choose an infinite compact connected metric space $X_I$ such that $X_{I}$ is not homeomorphic to $X_{I'}$ when $I \neq I'$. In \cite{Th2} this was done by taken spaces with different covering dimensions. Let $X_{I_0}$ be the one-point set. For each $I \in \mathbb I$ let $S_I$ be the Bauer simplex of Borel probability measures on $X_I$. Apply Corollary \ref{25-10-18k} with these choices and note that the simplexes $S^{\alpha}_{\beta}$ are Bauer and the extreme boundary $\partial_e S^{\alpha}_{\beta}$ is  homeomorphic to the one-point compactification of the topological disjoint union \eqref{30-10-18b}. Since $S^{\alpha}_{\beta}$ is only affinely homeomorphic to $S^{\alpha}_{\beta'}$ if their extreme boundaries are homeomorphic, the choice of spaces implies that this only happens if $\mathbb I_{\beta} = \mathbb I_{\beta'}$. By the choice of intervals this implies that $\beta = \beta'$.

\qed

 If $U$ and $\alpha$ are as in Theorem \ref{CAR2} and $\gamma$ is the one-parameter group on the $C^*$-algebra $B$ from Theorem 3.2 in \cite{BEH} applied with the appropriate subset of $[0,1]$, the tensor product action $ \alpha' = \alpha \otimes \gamma$ on $A = U\otimes B$ will have the properties specified in the following theorem.
 
 \begin{thm}\label{05-08-18a} Let $F$ be a closed subset of the real numbers $\mathbb R$. There is a simple unital $C^*$-algebra with a continuous one-parameter group $\alpha'$ of automorphisms such that the Choquet simplex $S^{\alpha'}_{\beta}$ of $\beta$-KMS states for $\alpha'$ is non-empty if and only if $\beta \in F$, and for $\beta,\beta' \in F$ the simplexes $S^{\alpha'}_{\beta}$ and $S^{\alpha'}_{\beta'}$ are only affinely homeomorphic to each other when $\beta = \beta'$.
\end{thm}
 
The proof is the same as the proof of Theorem 12.1 in \cite{Th2} and will not be repeated. For non-empty $F$ the algebra $A$ in Theorem \ref{05-08-18a} is a corner in a crossed product $B \rtimes \mathbb Z$ with $B$ a simple non-unital AF algebra. 

It was shown by Matui and Sato in the proof of Theorem 6.3 in \cite{MS} that there is a one-parameter group of automorphisms without ground states on any simple unital infinite-dimensional AF algebra with a unique trace state, thus disproving the Powers-Sakai conjecture, \cite{PS}. For these actions there must be an upper bound on the inverse temperatures, i.e. for all sufficient large $\beta$ there are no $\beta$-KMS states. In view of this it is natural to wonder if it is really necessary to leave the realm of UHF algebras to realize a field of KMS states as in Theorem \ref{05-08-18a}. More explicitly:

\begin{q} Which closed subsets of $\mathbb R$ can be realized as the set of inverse temperatures for a one-parameter group of automorphisms on a UHF algebra ?
\end{q}

\end{document}